\documentclass[10pt,a4paper]{amsart}

\usepackage{amsmath, amsthm, amssymb, amscd}
\usepackage{graphicx,color}
\usepackage{texdraw}
\usepackage{setspace}
\usepackage{mathrsfs}
\usepackage{cite}
\newtheorem{lemma}{Lemma}
\newtheorem{proposition}{Proposition}
\newtheorem{theorem}{Theorem}

\newtheorem{remark}{Remark}

\newcommand{\R}{\mathbb{R}}
\newcommand{\N}{\mathbb{N}}

\newcommand{\eps}{\varepsilon}
\renewcommand{\H}{\mathring{H}^{\frac{s}{2}}(\mathbb{R}^{N})}
\newcommand{\HO}{\mathring{H}^{\frac{s}{2}}(\Omega)}

\newcommand{\cC}{{\mathcal C}}

\newcommand{\cM}{{\mathcal M}}

\newcommand{\cP}{{\mathcal P}}

\newcommand{\cS}{{\mathcal S}}

\newcommand{\Ds}{(-\Delta)^{s/2}}

\DeclareMathOperator{\id}{id}

\title{Remainder terms in the fractional Sobolev Inequality}

\author{Shibing Chen \and Rupert L. Frank \and Tobias Weth}
\address{Department of Mathematics, University of Toronto, 40 St. George
  Street, Toronto, Ontario, Canada M5S 2E4}
\email{shibing.chen@utoronto.ca}
\address{Department of Mathematics, Fine Hall, Princeton University,
Princeton, NJ 08544 (USA)}
\email{rlfrank@math.princeton.edu}
\address{Institut f\"ur Mathematik, Goethe-Universit\"at, Robert-Mayer-Str.\ 10, 60054 Frankfurt (Germany)}
\email{weth@math.uni-frankfurt.de}
\date{\today}

\begin{document}

\begin{abstract}
We show that the fractional Sobolev inequality for the embedding $\H
\hookrightarrow L^{\frac{2N}{N-s}}(\R^N)$, $s \in (0,N)$ can be sharpened by 
adding a remainder term proportional to the distance to the set of optimizers. As a corollary, we derive the
existence of a remainder term in the weak $L^{\frac{N}{N-s}}$-norm for
  functions supported in a domain of finite measure. Our results
  generalize earlier work for the non-fractional case 
where $s$ is an even integer.
\end{abstract}

 \maketitle

\section{Introduction}
\label{sec:introduction}

In this note we consider the fractional Sobolev inequality
\begin{equation}
  \label{eq:1}
\|u\|_{s/2}^2 \geq
\cS \left(\int_{\mathbb{R}^{N}}|u|^{q}dx\right)^{\frac{2}{q}}
\qquad \text{for all $u\in \mathring{H}^{\frac{s}{2}}(\mathbb{R}^{N})$,} 
\end{equation}
where $0<s<N$, $q=\frac{2N}{N-s}$, and $\mathring{H}^{\frac{s}{2}}(\mathbb{R}^{N})$ is the space of all
tempered distributions $u$ such that  
$$
\hat u \in L^1_{loc}(\R^N) \qquad \text{and} \qquad  \|u \|_{s/2}^2: = \int_{\mathbb{R}^{N}}|\xi|^s|\hat u|^2 dx < \infty.
$$ 
Here, as usual, $\hat u$ denotes the (distributional)
Fourier transform of $u$. The best Sobolev constant 
\begin{equation}
  \label{eq:21}
\cS=\cS(N,s)= 2^{s} \pi^{\frac{s}{2}} \frac{\Gamma(\frac{N+s}{2})}{
  \Gamma(\frac{N-s}{2})} \Bigl(\frac{\Gamma(\frac{N}{2})}{\Gamma(N)}\Bigr)^{s/N}, 
\end{equation}
i.e., the largest possible constant in (\ref{eq:1}), has been
computed first in the special case $s=2$, $N=3$ by Rosen \cite{Rosen}
and then independently by Aubin \cite{Aubin}
and Talenti \cite{Talenti} for $s=2$ and all dimensions $N$. For general $s \in (0,N)$,
the best constant has been given by Lieb \cite{L} for an
equivalent reformulation of inequality (\ref{eq:1}), the (diagonal) Hardy-Littlewood-Sobolev
inequality. In order to discuss this equivalence in some more detail, we note that
\begin{equation}
  \label{eq:2}
\|u \|_{s/2}^2 = \int_{\mathbb{R}^{N}}u(-\Delta)^{s/2}u dx 
\end{equation}
for every Schwartz
function $u$, where the operator $(-\Delta)^{s/2}$ is defined by  
$$
\widehat {\Ds u}(\xi)=|\xi|^{s} \widehat u (\xi) \qquad \text{for a.e.
  $\xi \in \R^N$.} 
$$
Moreover, $\mathring{H}^{\frac{s}{2}}(\mathbb{R}^{N})$ is also given
as the completion of smooth functions with compact support under the
norm $\|\cdot \|_{s/2}$. The (diagonal) Hardy-Littlewood-Sobolev inequality states that 
\begin{equation}
  \label{eq:4}
\Bigl|\int_{\R^N} \int_{\R^N} \frac{f(x)
  g(y)}{|x-y|^\lambda}\,dx\,dy\Bigr| \le 
\pi^{\lambda/2}
\frac{\Gamma(\frac{N-\lambda}{2})}{\Gamma(N-\frac{\lambda}{2})}
\Bigl(\frac{\Gamma(N)}{\Gamma(N/2)}\Bigr)^{1-\frac{\lambda}{N}}
|f|_p |g|_p   
\end{equation}
for all $f,g \in L^p(\R^N)$, where $0<\lambda<N$ and
$p=\frac{2N}{2N-\lambda}$. Here and in the following, we let
$|\cdot|_r$ denote the usual $L^r$-norm for $1 \le r\le \infty$. The equivalence of (\ref{eq:1}) and
(\ref{eq:4}) follows -- by a duality argument -- from the fact 
that for every $f \in L^{\frac{q}{q-1}}(\R^N)$ there exists a unique
solution $(-\Delta)^{-s/2}f \in \mathring{H}^{\frac{s}{2}}(\mathbb{R}^{N})$ of the
equation $\Ds u = f$ given by convolution with the Riesz potential,
i.e., by 
\begin{equation}
  \label{eq:5}
[(-\Delta)^{-s/2}f](x) = 2^{-s}\pi^{-\frac{N}{2}}
\frac{\Gamma(\frac{N-s}{2})}{\Gamma(s/2)}\int_{\R^N}
\frac{1}{|x-y|^{N-s}}f(y)\,dy 
\qquad \text{for a.e. $x \in
  \R^N$.}
\end{equation}
In \cite{L}, Lieb identified the extremal functions for
(\ref{eq:4}), and his results imply that equality
holds in (\ref{eq:1}) for nontrivial $u$ if and only if $u$ is contained in an
$N+2$-dimensional submanifold $\cM$ of $\H$ given as the set of
functions which, up to
translation, dilation and multiplication by a nonzero constant,  
coincide with    
\begin{equation}
  \label{eq:6}
U \in \H, \qquad U(x)=(1+|x|^2)^{-\frac{N-s}{2}}.
\end{equation}
For the special case $s=2$, i.e., the first order Sobolev inequality,
Brezis and Lieb \cite{BL} asked the question whether a remainder term
-- proportional to the
quadratic distance of the function $u$ to the manifold $\cM$ -- can be
added to the right hand side of (\ref{eq:1}). This question was
answered affirmatively in the case $s=2$ by Bianchi and Egnell
\cite{BE}, and their result was extended later to the case 
$s=4$ in \cite{LW} and to the case of an arbitrary even positive integer $s<N$ in
\cite{BWW}. The main purpose of the present note is to obtain a
corresponding remainder term
inequality for  all (real) values $s \in (0,N)$. Our main result is the following.

\begin{theorem}
\label{maintheorem}
Let 
\begin{equation}
  \label{eq:7}
\cM:= \Bigl\{ c\, U\bigl(\frac{\cdot - x_0}{\eps}\bigr) \,:\, c \in \R
\setminus \{0\}, x_0 \in \R^N,
\eps>0\Bigr \} \: \subset \: \H, 
\end{equation}
where $U$ is defined in (\ref{eq:6}). Then there exists a positive constant $\alpha$ depending only on the dimension $N$ and $s\in(0,N)$ such that
\begin{equation}
  \label{eq:10}
{d}^{2}(u, {\cM})\geq\int_{\mathbb{R}^{N}}u(-\Delta)^{s/2}(u)dx- \cS \left(\int_{\mathbb{R}^{N}}|u|^{q}dx\right)^{\frac{2}{q}}\geq \alpha\, {d}^{2}(u, \cM)
\end{equation}
for all $u\in \H$, where ${d}(u, {\cM})=\min\{\|u-\varphi\|_{s/2}\::\: \varphi\in{\cM} \}.$
 \end{theorem}

We briefly explain the strategy to prove this remainder term
inequality which goes back to Bianchi and Egnell \cite{BE} in the case
$s=2$. First, the inequality is proved in a small neighborhood of
the optimizer $U \in \cM$ defined in (\ref{eq:6}). Considering a second order Taylor
expansion of the difference functional 
$$
u \mapsto \Phi(u):=\|u\|_{s/2}^2 - \cS \left(\int_{\mathbb{R}^{N}}|u|^{q}dx\right)^{\frac{2}{q}},
$$
at $U$, it is not dificult to see that (\ref{eq:10}) holds in a
neighborhood of $U$ with some $\alpha>0$ if and only if the second
derivative $\Phi''(U)$ is positive definite on the
$(N-2)-$codimensional normal space to the
manifold $\cM$ at $U$. This normal non-degeneracy property is the
crucial step in the argument. Once inequality (\ref{eq:10}) is established in 
a neighborhood of $U$, it extends to a neighborhood of the whole manifold
$\cM$ by as a consequence of the conformal invariance of all terms in
(\ref{eq:7}). We will recall this conformal invariance in detail in Section~\ref{sec:preliminaries} below. Finally, to obtain the global
version of (\ref{eq:10}), a concentration compactness type argument is
applied to show that sequences $(u_n)_n$
in $\H$ with $\Phi(u_n) \to 0$ as $n \to \infty$ satisfy $d(u_n, \cM) \to 0$ as
$n \to \infty$.\\
The general idea described here had already been used in 
\cite{BE,LW,BWW}, but the proofs of the normal non-degeneracy property
in these papers strongly rely on the assumption that $s$ is an even
positve integer
and therefore the eigenvalue problem for $\Phi''(U)$ can be written as
a differential equation. In particular, ODE arguments are used to study the radial part of the
corresponding eigenvalue problem. This method does not apply for
general $s \in (0,N)$. On the other hand, one may observe that the
eigenvalue problem has a much simpler form once inequality
(\ref{eq:10}) is pulled back on the unit sphere $\mathbb{S}^N \subset \R^{N+1}$ via stereographic 
projection. The equivalent version of Theorem~\ref{maintheorem} on 
$\mathbb{S}^N$ is given in Theorem~\ref{maintheorem-reform} below. The
idea of studying (\ref{eq:1}) in its equivalent form on ${\mathbb{S}}^N$ also goes back to Lieb's paper \cite{L}
where the (equivalent) Hardy-Littlewood-Sobolev inequality was
considered. Afterwards it has been applied in many related problems
dealing with Sobolev type inequalities and corresponding Euler-Lagrange
equations, see e.g. \cite{BSW,D,M,Beckner:93} and the references therein. To our
knowledge, its usefulness
to identify remainder terms has not been noted so far.\\
About twenty years after the seminal work of Bianchi and
Egnell \cite{BE}, the topic of remainder terms in first order Sobolev
inequalities (and isoperimetric inequalities) 
 has again attracted a lot of attention in the last years. The recent works
use techniques from symmetrization (see, e.g., \cite{CFMP,FMP}),
optimal transportation (see, e.g., \cite{FiMaPr}), and fast diffusion
(see, e.g., \cite{Do,DoTo,JiXi}); see also \cite{CaFi} for a recent
application of remainder terms. However, while these new methods lead to explicit
constants and allow to treat non-Hilbertian Sobolev norms, the estimates for the remainder terms are typically weaker
than in the result of Bianchi and Egnell. It is not clear
to us to which extent the symmetrization and the optimal
transportation approach can be extended to give remainder terms in the
higher order case or in the case of arbitrary real powers of the
Laplacian (see \cite{JiXi} for a fast diffusion approach in the
fractional case). We therefore think it is remarkable that the
original strategy of Bianchi--Egnell can be generalized to the full
family of conformally invariant Hilbertian Sobolev inequalities.\\
As a corollary of Theorem~\ref{maintheorem}, we also derive a
remainder term inequality for the function space
$\HO$ which -- for a subdomain $\Omega \subset \R^N$ -- is defined as the completion of
$\cC_0^\infty(\Omega)$ with respect to the norm
$\|\cdot\|_{s/2}$. In the case where $\Omega$ has a continuous
boundary, we have 
$$
\HO= \{u \in H^s(\Omega)\::\: \text{$\tilde u \in \H$}\},
$$
where $\tilde u$ denotes the trivial extension of a function $u \in
H^s(\Omega)$ on $\R^N$. We also recall that, for $1 < r < \infty$, the weak $L^r$-norm
of a measurable function $u$ on $\Omega$ is given by 
$$
|u|_{w,r,\Omega}= \sup_{\stackrel{A \subset \Omega}{|A|>0}}|A|^{\frac{1}{r}-1} \int_A |u|\,dx,
$$
see e.g. \cite{hunt}. 

\begin{theorem}
\label{maincorollary}  
Let, as before, $q=\frac{2N}{N-s}$. Then there exists
a constant $C>0$ depending only on $N$ and $s \in (0,N)$ such that for
every domain $\Omega \subset \R^N$ with $|\Omega|<\infty$ and every $u \in \HO$ 
we have
\begin{equation}
  \label{eq:32}
\|u\|_{s/2}^2 -
\cS \left(\int_{\Omega}|u|^{q}dx\right)^{\frac{2}{q}} \ge C |\Omega |^{-\frac{2}{q}}
|u|_{w,q/\! 2,\Omega}^2
\end{equation}
\end{theorem}

For fixed bounded domains $\Omega \subset \R^N$, the existence of a weak
$L^{q/2}$-remainder term is due to Brezis and Lieb \cite{BL} in the case $s=2$ and
to Gazzola and Grunau \cite{GG} in the case of an arbitrary even positive
integer $s<N$. Bianchi and Egnell
\cite{BE} gave an alternative proof in the case $s=2$ using the corresponding special
case of inequality (\ref{eq:10}). We will
follow similar ideas in our proof of
Theorem~\ref{maincorollary}, using Theorem~\ref{maintheorem} in full
generality. We note that some additional care is needed to get a remainder term which
only depends on $|\Omega|$ and not on $\Omega$ itself.  

The paper is organized as follows. In Section~\ref{sec:preliminaries}
we recall the conformal invariance of the problem, and we discuss the framework
for an equivalent version of Theorem~\ref{maintheorem} on the sphere
$\mathbb{S}^N \subset \R^{N+1}$, see Theorem~\ref{maintheorem-reform}. In
Section~\ref{sec:proof-remainder-term} we prove this Theorem, thus
completing the proof of Theorem~\ref{maintheorem}. In
Section~\ref{sec:weak-lq2-remainder} we give the proof of
Theorem~\ref{maincorollary}.   
   
We conclude by pointing out the open problem to find an explicit constant
$\alpha>0$ in (\ref{eq:10}) via a constructive proof of Theorem~\ref{maintheorem}. For a local version of
Theorem~\ref{maintheorem} where the right hand side of (\ref{eq:10})
is replaced by $\alpha d^2(u,\cM) + o(d^2(u,\cM))$ and only $u \in \H$
with $d(u,\cM)<\|u\|_{s/2}$ is considered, the best constant is $\alpha=
\frac{2s}{N+s+2}$. This follows from 
Proposition~\ref{sec:proof-remainder-term-1} below.

\section{Preliminaries}
\label{sec:preliminaries}

In the following, we will denote the scalar product in $\H$ by 
$$
\langle u,v
\rangle_{s/2} = \int_{\R^N} |\xi|^s \hat u(\xi) \overline{\hat v(\xi)}\,d\xi,
$$
so that $\|u\|_{s/2}^2 = \langle u,u \rangle_{s/2}$ for $u \in \H$. In
the remainder of this section, $0<s<N$ is fixed and we abbreviate $q=
2n/(N-s)$. We recall that the group of conformal transformations on
$\R^N$ is  generated by translations, rotations, dilations and the inversion $x \mapsto
\frac{x}{|x|^2}$. If $h$ is one of these transformations with Jacobian determinant $J_h$, then 
for any functions $u,v \in \mathring{H}^{\frac{s}{2}}(\mathbb{R}^{N})$
we have $J_h^{\frac{1}{q}} u \circ h, J_h^{\frac{1}{q}} v \circ h
\in \mathring{H}^{\frac{s}{2}}(\mathbb{R}^{N})$ and
\begin{equation}
  \label{eq:8}
\langle  J_h^{\frac{1}{q}} u \circ h , J_h^{\frac{1}{q}} v \circ
h \rangle_{s/2}=\langle u,v \rangle_{s/2}.
\end{equation}
This property is a consequence of the conformal covariance of
the operator $\Ds$, i.e. of the equality 
\begin{equation}
  \label{eq:11}
\Ds (J_h^{\frac{1}{q}} u \circ h) = J_h^{\frac{N+s}{2N}} [\Ds u]
\circ h    
\end{equation}
for all conformal transformations $h$ on $\R^N$ and all Schwartz
functions $u$. As stated in \cite[Proposition 2.1]{M}, (\ref{eq:11})
is most easily derived by considering the inverse operator
$(-\Delta)^{-s/2}$ given in (\ref{eq:5}). Indeed, the identity 
\begin{equation}
\label{eq:12}
(-\Delta)^{-s/2} (J_h^{\frac{N+s}{2N}} u \circ h) =
J_h^{\frac{1}{q}} [(-\Delta)^{-s/2} u] \circ h    
\end{equation}
is equivalent to (\ref{eq:11}), and it can be verified case by
case for dilations, rotations,
translations and the inversion. In the latter form related to the Riesz potential, the
conformal covariance had already been used by Lieb in \cite{L}.

Note that, if $h$ is a conformal transformation on $\R^n$, it follows
from (\ref{eq:8}) that the map 
$u \mapsto J_h^{\frac{1}{q}} u \circ h$ preserves distances with respect to the
norm $\|\cdot\|_{s/2}$, i.e. we have 
\begin{equation}
\label{eq:9}
\|J_h^{\frac{1}{q}} u \circ h - J_h^{\frac{1}{q}} v \circ h\|_{s/2}=\| u
-v \|_{s/2} \qquad \text{for all $u,v \in \H$.}
\end{equation}
Since the set $\cM$ is also invariant under the transformations $u
\mapsto J_h^{\frac{1}{q}} u \circ h$, we conclude that 
$d(J_h^{\frac{1}{q}} u \circ h, \cM)=d(u,\cM)$ for all $u \in \H$. We
also note that 
\begin{equation}
  \label{eq:3}
|J_h^{\frac{1}{q}} u \circ h|_{q} = |u|_{q} \qquad \text{for any $u
  \in L^q(\R^N)$}
\end{equation}
and any conformal transformation $h$ on $\R^N$, which follows by an
easy computation. In the following, we consider the inverse
stereographic projection  
$$
\pi: \R^N \to \mathbb{S}^{N} \subset \R^{N+1}, \qquad
\pi(x)=(\frac{2x}{1+|x|^2},\frac{1-|x|^2}{1+|x|^2}).
$$
We recall that $\pi$ is a conformal diffeomorphism. 
More precisely, if $g_{\R^N}$ denotes the flat euclidian metric 
on $\R^N$ and $g_{\mathbb{S}^{N}}$ denotes the metric induced by the
embedding $\mathbb{S}^{N}  \subset \R^{N+1}$, then the pullback of $g_{\mathbb{S}^{N}}$ to $\R^N$ satisfies
\begin{equation}
\label{conformfactor}
{\pi}^*g_{\mathbb{S}^{N}}=\frac{4}{(1+|\cdot|^2)^2}g_{\R^N}.
\end{equation}
Moreover, the corresponding volume element is given by 
\begin{equation}
  \label{eq:13}
J_\pi(x) dx = \Bigl(\frac{2}{1+|x|^2}\Bigr)^N dx,
\end{equation}
For a function $v:{\mathbb{S}}^N \to \R$, we may now define 
$$
\cP v: \R^N \to \R, \qquad [\cP v](x)= J_\pi(x)^{\frac{1}{q}} v(\pi(x)) = 
\Bigl(\frac{2}{1+|x|^2}\Bigr)^{\frac{N-s}{2}} v(\pi(x)).
$$
From (\ref{eq:13}), it is easy to see that $\cP$ defines an isometric
isomorphism between $L^q(\mathbb{S}^N)$ and $L^q(\R^N)$. We also note that 
\begin{equation}
  \label{eq:15}
\cP\, 1 = 2^{(N-s)/2} U,  
\end{equation}
where $1$ stands for unit function on $\mathbb{S}^N$ and $U$ is defined in (\ref{eq:6}). Moreover, $H^{\frac{s}{2}}(\mathbb{S}^{N})$ is the completion  
of the space of smooth functions on ${\mathbb{S}}^N$ 
under the norm $\|\cdot \|_{*}$ induced by scalar product
$$
(u,v) \mapsto \langle u,v \rangle_*= \langle \cP u, \cP v \rangle_{s/2}.
$$
We will always consider $H^{\frac{s}{2}}(\mathbb{S}^{N})$ with the
norm $\|\cdot\|_*$ induced by this scalar product (for
matters of convenience, we suppress the dependence on
$s$ at this point). Hence, by construction, 
$$
\text{$\cP$ is also an isometric isomorphism
  $(H^{\frac{s}{2}}(\mathbb{S}^{N}),\|\cdot\|_*)  \to
  (\H,\|\cdot\|_{s/2})$.}
$$
Next we note that $\langle \cdot,\cdot \rangle_*$ is the quadratic
form of a unique positive self adjoint operator in $L^2(\mathbb{S}^N)$ which is commonly denoted by $A_s$ in the literature. This operator is formally given by 
$$
[A_s w] \circ \pi = J_\pi^{-\frac{N+s}{2N}} \Ds (\cP w).
$$
A key ingredient of the proof of Theorem~\ref{maintheorem} is the
following representation of $A_s$ as a function of the
Laplace-Beltrami Operator $\Delta_{{\mathbb{S}}^N}$
on ${\mathbb{S}}^N$: 
\begin{equation}
  \label{eq:14}
A_s= \frac{\Gamma(B+\frac{1+s}{2})}{\Gamma(B+\frac{1-s}{2})} \qquad
\text{with $B= \sqrt{-\Delta_{{\mathbb{S}}^N} + \bigl(\frac{N-1}{2}\bigr)^2}$.}
\end{equation}
This formula is most easily derived by considering the inverse of $A_s$ and using
the Funk-Hecke formula, see \cite{Beckner:93} and also \cite{M}. It
also shows that the domain of $A_s$
coincides with $H^s({\mathbb{S}}^N)$. The following statement
is a 
mere reformulation of (\ref{eq:14}). 

\begin{proposition}
\label{lemma1}
The operator $A_{s}$ is self adjoint and has compact resolvent. Its
spectrum is given as the sequence of eigenvalues 
$$
\lambda_k(s)= \frac{\Gamma(\frac{N+s}{2}+k)}{\Gamma(\frac{N-s}{2}+k)},
\qquad k \in \N_0,
$$
and the eigenspace corresponding to the eigenvalue $\lambda_k(s)$ is
spanned by the spherical harmonics $Y_{k,j},\, j= 1,\dots, {k+N
  \choose N} - {k+N-2
  \choose N}$, of degree $k$.
\end{proposition}

Next, we note that, via the isometric isomorphism $\cP$, inequality
(\ref{eq:1}) is equivalent to 
\begin{equation}
  \label{eq:16}
\|u\|_*^2 \geq
\cS |u|_q^2 \qquad \text{for all $u\in H^{\frac{s}{2}}(\mathbb{S}^{N})$,} 
\end{equation}
with $q=\frac{2N}{N-s}$. Here, in accordance with the previous notation, we also write $|\cdot|_r$ for the $L^r$-norm of
a function in $L^r(\mathbb{S}^N)$, $1 \le r \le \infty$. Equality is
attained in (\ref{eq:16}) for nontrivial $u$ if and only if $u \in
\cM_*$, where 
$$
\cM_*:= \cP^{-1}(\cM)= \{v \in H^{\frac{s}{2}}(\mathbb{S}^{N})\::\:
\cP v \in \cM\}. 
$$
Moreover, the remainder term inequality (\ref{eq:10}) is equivalent to 
\begin{equation}
\label{eq:17}
{d}^{2}(u, {\cM_*})\geq \|u\|_*^2-
\cS|u|_q^{2}\geq
\alpha\, {d}^{2}(u, \cM_*)\quad \text{for $u\in {H}^{s/2}(\mathbb{S}^{N})$},
\end{equation}
where ${d}(u, {\cM}_*)=\min\{\|u-\varphi\|_{*}\::\:\varphi\in{\cM} \}.$ 
We may therefore reformulate Theorem~\ref{maintheorem} as follows.
\begin{theorem}
\label{maintheorem-reform}
There exists a positive constant $\alpha$ depending only on the dimension $N$ and $s\in(0,N)$ such that
(\ref{eq:17}) holds.
 \end{theorem}

We will prove Theorem~\ref{maintheorem-reform} in
Section~\ref{sec:proof-remainder-term} below, thus completing the
proof of Theorem~\ref{maintheorem}. We close this section with some
comments on the conformal invariance of the reformulated problem and
the geometry of $\cM_*$. Via
stereographic projection, the conformal transformations on
${\mathbb{S}}^N$ are in 1-1-correspondance with the conformal
transformations on $\R^N$. So, if $\tau$ is an element of the
conformal group of ${\mathbb{S}}^{N}$ with Jacobian determinant $J_\tau$, then (\ref{eq:3}) and
(\ref{eq:8}) imply that 
\begin{equation}
  \label{eq:3_*}
\langle J_\tau^{\frac{1}{q}} u \circ \tau, J_\tau^{\frac{1}{q}} v
\circ \tau \rangle_{s/2} = \langle u,v \rangle_* \qquad \text{and}\qquad  
|J_h^{\frac{1}{q}} u \circ h|_{q} = |u|_{q}
\end{equation}
for all $u,v \in H^{\frac{s}{2}}(\mathbb{S}^{N})$. From
(\ref{eq:15}), we deduce the representation 
$$
\cM_*=\{cJ_{\tau}^{\frac{1}{q}}|\ \tau\ \text{is an element of the
  conformal group of}\ S^{N}, c \in \mathbb{R} \setminus \{0\} \}.
$$ 
Since the Jacobian determinant $J_\tau$ of a conformal transformation
$\tau$ on ${\mathbb{S}}^N$ has the form $J_{\tau}(\xi)=(1-\xi\cdot\theta)^{-n}$
for some $\theta\in B^{n+1}:=\{x \in \R^{N+1}\::\: |x|<1\}$, $\cM_*$ can be viewed as an $N+2$ dimensional smooth manifold embedded in $H^{\frac{s}{2}}(\mathbb{S}^{N})$ via the mapping
\begin{eqnarray}
\R \setminus \{0\} \times B^{N+1}\,  \to\, H^{\frac{s}{2}}(\mathbb{S}^{N}),\qquad (c, \theta)\,\mapsto\, u_{c,\theta},
\end{eqnarray}
where $u_{c,\theta}(\xi)=c(1-\xi\cdot\theta)^{-\frac{N-s}{2}}$ for
$\xi \in {\mathbb{S}}^N$.
This immediately implies that the tangent space $T_{1}\cM_*$ at the
function $1=u_{1,0}$ is
generated by the spherical harmonics $Y_0^0=1$ and $Y_1^j$,
$j=1,\dots,N+1$, given by 
$$
Y_1^j(\xi)= \xi_j \qquad \text{for $\xi=(\xi_1,\dots,\xi_{N+1}) \in
  \mathbb{S}^N \subset \R^{N+1}$.}
$$
Hence $T_{1}\cM_*$ coincides precisely
with the generalized eigenspace of the operator $A_s$ corresponding to
the eigenvalues $\lambda_0(s)$ and $\lambda_1(s)$. Combining this fact
with the minimax characterization of the eigenvalue $\lambda_2(s)$, we readily deduce that  
\begin{equation}
\label{eq:30}
\lambda_2(s)=\mathop{\text{inf}}_{v\in
  T_{1}\cM_*^\perp}\frac{\|v\|^2}{|v|_2^2}
\end{equation}
with 
\begin{equation}
  \label{eq:31}
T_{1}\cM_*^\perp:= \{v \in H^{\frac{s}{2}}(\mathbb{S}^{N})\::\:
\langle v, w \rangle_* =0 \; \text{for all $w \in T_{1}\cM_*$}\}.  
\end{equation}
The identity (\ref{eq:30}) will be of crucial importance for the local
verification of (\ref{eq:17}) close to the manifold $\cM_*$. 

\section{Proof of the remainder term inequality on the sphere}
\label{sec:proof-remainder-term}

We first prove a local variant of Theorem~\ref{maintheorem-reform}. 

\begin{proposition}
\label{sec:proof-remainder-term-1}
For all $u\in H^{\frac{s}{2}}(\mathbb{S}^{N})$ with
$d(u,\cM_*)<\|u\|_*$, we have 
\begin{equation}
  \label{eq:29}
d^2(u,\cM_*) \ge \|u\|_*^2 -\cS |u|_q^2
\geq \frac{2s}{N+s+2} d^{2}(u,\cM_*)+o(d^{2}(u,\cM_*)).
\end{equation}
\end{proposition}

\begin{proof}
We consider the functional 
\begin{equation}
  \label{eq:18}
\Psi: H^{\frac{s}{2}}(\mathbb{S}^{N}) \to \R,\qquad \Psi(u)= \|u\|_*^2-
\cS |u|_q^2.  
\end{equation}
It is easy to see that $\Psi$ is of class $\cC^2$ on
$H^{\frac{s}{2}}(\mathbb{S}^{N}) \setminus \{0\}$. Moreover,
\begin{equation}
  \label{eq:19}
\Psi'(u)v = 2 \langle u,v \rangle_*- 2 \cS |u|_q^{2-q} \int_{{\mathbb{S}}^N} |u|^{q-2}u
v\,d\xi  
\end{equation}
and 
\begin{align}
\nonumber 
\frac{1}{2}\Psi''(u)(v,w) = \langle v,w \rangle_*-&\cS (2-q) |u|_q^{2-2q} 
\int_{{\mathbb{S}}^N} |u|^{q-2}u v\,d\xi\, \int_{{\mathbb{S}}^N} |u|^{q-2}u w\,d\xi\\
-&\cS (q-1) |u|_q^{2-q} \int_{{\mathbb{S}}^N} |u|^{q-2}v w\,d\xi
  \label{eq:20}
\end{align}
for $u \in H^{\frac{s}{2}}(\mathbb{S}^{N}) \setminus \{0\}$, $v,w \in
H^{\frac{s}{2}}(\mathbb{S}^{N})$.\\
Next, let $u\in
H^{\frac{s}{2}}(\mathbb{S}^{N})$ with $d(u,\cM_*)<\|u\|_*$. It is easy
to see that  
$d(u, \cM_*)$ is achieved by some
function $c J_{\tau}^{\frac{1}{q}}$ in $\cM_*$ with $c \in \R \setminus \{0\}$
and a conformal transformation $\tau$ on ${\mathbb{S}}^N$. Replacing $u$ with
$\frac{1}{c} J_{\tau^{-1}}^{\frac{1}{q}} u \circ \tau^{-1}$ and using 
(\ref{eq:3_*}), we may assume that $c=1$ and $\tau= \id$, hence
we may write $u=1+v$ with $v \in T_1 \cM_*^\perp$, the normal space of $\cM_*$ at
$1$ defined in (\ref{eq:31}), and $d(u,\cM_*)= \|v\|_*.$ We note that $\Psi(1)=0$ and 
$\Psi'(1)=0$ (since the function $1$ is a global minimizer of $\Psi$). Moreover, the condition $v \in T_1 \cM_*^\perp$ in particular implies -- since $1 \in T_1\cM_*$ -- that 
\begin{equation}
  \label{eq:orthogonal}
\langle 1,v \rangle_*=0 \qquad \text{and}\qquad \int_{{\mathbb{S}}^N} v\,d \xi = 0.  
\end{equation}
In particular, we find that  
\begin{align*}
\Psi(u)&=\Psi(1+v)= \|1\|_*^2+ \|v\|_*^2-\cS |1+v|_q^2 \le  \|1\|_*^2+ \|v\|_*^2-\cS |\mathbb{S}^N|^{\frac{2-q}{q}} |1+v|_2^2\\  
&= \|1\|_*^2+ \|v\|_*^2 -\cS |\mathbb{S}^N|^{\frac{2-q}{q}} (|\mathbb{S}^N|+|v|_2^2)=
\Psi(1)+  \|v\|_*^2 - \cS |\mathbb{S}^N|^{\frac{2-q}{q}}|v|_2^2\\
&\le
\|v\|_*^2= d^2(u,\cM_*), 
\end{align*}
and this yields the first inequality in (\ref{eq:29}). Moreover, from
(\ref{eq:20}) and (\ref{eq:orthogonal}) we infer that   
\begin{equation*}
\frac{1}{2}\Psi''(1)(v,v) = \|v\|_*^2-  (q-1) \cS
|{\mathbb{S}}^N|^{\frac{2-q}{q}} \int_{{\mathbb{S}}^N}v^2\,d\xi.
\end{equation*}
A second order Taylor expansion of $\Psi$ at $1$ thus yields 
\begin{align*}
\Psi(u)= \Psi(1+v)&= \frac{1}{2}\Psi''(1)(v,v) +o(\|v\|_*^2)\\
& =\|v\|_*^2 - (q-1) \cS |{\mathbb{S}}^N|^{\frac{2-q}{q}}|v|_2^2 +o(\|v\|_*^2).
\end{align*}
Using (\ref{eq:21}) and the identity $|{\mathbb{S}}^N|= 2\pi^{\frac{N+1}{2}} \Gamma(\frac{N+1}{2})^{-1}$, 
we find by a short computation (using the duplication formula for the
Gamma function) that 
$$
(q-1) \cS
  |{\mathbb{S}}^N|^{\frac{2-q}{q}}= \frac{N+s}{N-s}\: \cS |{\mathbb{S}}^N|^{-\frac{s}{N}}=
  \frac{\Gamma(\frac{N+s}{2}+1)}{\Gamma(\frac{N-s}{2}+1)} = \lambda_1(s).
$$
Noting moreover that $|v|^2_2 \le \frac{\|v\|_*^2}{\lambda_2(s)} $ as a consequence
of (\ref{eq:30}), we conclude that 
$$
\Psi(u) \ge \|v\|_*^2 \Bigl( 1 -
\frac{\lambda_1(s)}{\lambda_2(s)}+o(1) \Bigr)= d(u,\cM_*)^2 \Bigl(
\frac{2s}{N+s+2} +o(1) \Bigr)
$$
This shows the second inequality in (\ref{eq:29}).
\end{proof}

The next tool we need is the following property of optimizing
sequences for (\ref{eq:1}).

\begin{lemma}
\label{sec:proof-remainder-term-2} 
Let $(u_m)_m \subset \H \setminus \{0\}$ be a sequence with  
$\lim \limits_{m \to \infty} \frac{\|u_m\|_{*}^2}{|u_m|_q^2}= \cS$. Then $\frac{d(u_m,\cM_*)}{\|u_m\|_*} \to 0$ as
$m \to \infty$.
\end{lemma}

\begin{proof}
By homogeneity, we may assume that $\|u_m\|_*=1$ for all $m \in
\N$, and we need to show that $d(u_m,\cM_*) \to 0$ as $m \to \infty$. We let $v_m= \cP
u_m \in \H$ for $m \in \N$; then $\|v_m\|_{s/2}=1$ for all $m$, and  
\begin{equation}
  \label{eq:23}
\frac{1}{|v_m|_q^2} \to \cS
\qquad \text{as $m \rightarrow \infty$}.
\end{equation}
By the profile decomposition
theorem of G\'erard (see \cite[Th\'eor\`eme 1.1 and Remarque
1.2]{gerard:98}) , there exists a subsequence -- still denoted by
$(v_m)_m$ -- and\\ 
$\bullet$ a sequence $(\psi_j)_j$ of functions $\psi_j \in \H$,\\
$\bullet$ an increasing sequence of numbers $l_m \in \N$, $m \in \N$,\\
$\bullet$ a double sequence of values $h_m^j \in (0,\infty)$, $m,j \in \N$,\\
$\bullet$ a double sequence of points $x_m^j \in \R^N$, $m,j \in \N$\\
such that 
\begin{align}
&\Bigl|v_m- \sum_{j=1}^{l_m} \bigl(h_m^j\bigr)^{-\frac{s}{2q}}\,
\psi_j \bigl(\frac{\cdot - x_m^j}{h_m^j}\bigr)\Bigr|_q \to 0 \qquad
\text{as $m \to \infty$}, \label{eq:25}\\ 
& |v_m|_{q}^q \to \sum_{j=1}^\infty |\psi_j|_{q}^q
\quad \text{as $m \to
  \infty$} \qquad \text{and} \qquad \sum_{j=1}^\infty
\|\psi_j\|_{s/2}^2 \:\le\: 1.
 \label{eq:26}
\end{align}
Combining the Sobolev inequality~(\ref{eq:1}) with
(\ref{eq:26}) and using the
concavity of the function $t \mapsto t^{2/q}$, we find that 
\begin{equation}
  \label{eq:28}
1 \ge \cS \sum_{j=1}^\infty 
|\psi_j|_{q}^2 \ge \cS \Bigl(\sum_{j=1}^\infty 
|\psi_j|_{q}^q\Bigr)^{2/q}= \cS \lim_{m \to \infty} |v_m|_{q}^2.   
\end{equation}
By (\ref{eq:23}), equality holds in all steps in (\ref{eq:28}). The strict
concavity of the function $t \mapsto t^{2/q}$ then shows that $\psi_j
\equiv 0$ for all but one $j \in \N$, say, $j=1$, where 
$\cS |\psi_1|_{q}^2 = 1$ and $\|\psi_1\|_{s/2}=1$
as a consequence of (\ref{eq:26}),~(\ref{eq:28}) and the Sobolev
inequality (\ref{eq:1}). Hence $\Psi_1 \in \cM$, and from
(\ref{eq:25}) it now follows that 
$$
\Bigl|v_m- \bigl(h_m^1\bigr)^{-\frac{s}{2q}}\,
\psi_1 \bigl(\frac{\cdot - x_m^1}{h_m^1}\bigr)\Bigr|_q \to 0 \qquad
\text{as $m \to \infty$.} 
$$
Therefore, defining 
$$ 
\tilde v_m \in \H,\quad \tilde v_m(x)= \bigl(h_m^1\bigr)^{\frac{s}{2q}}
v_m(h_m^1 x + x_m^1) \qquad \text{for $m \in \N$,}
$$
we have $\tilde v_m  \to  \psi_1$ in $L^q(\R^N)$ for $m \to \infty$,
but then also $\tilde v_m \to \psi_1$ in $\H$ strongly since $\|\tilde
v_m\|_{s/2}=\|v_m\|_{s/2} =1=\|\psi_1\|_{s/2}$ for all $m \in \N$. Consequently, $d(\tilde v_m,\cM) \to 0$. By the invariance property 
(\ref{eq:9}), we then have $d(v_m,\cM) \to 0$ and therefore also 
$d(u_m,\cM_*) \to 0$ as $m \to \infty$, since $\cP$ is an isometry.  
\end{proof}

\begin{remark}{\rm 
(i) We note that we do not need the full strength of G\'erard's
profile decomposition theorem. Inductively, G\'erard writes $v_m$ as an
infinite sum of bubbles, see (\ref{eq:25}) and \cite{gerard:98}. For our proof it is enough
to stop this procedure after the very first step. As soon as one
bubble is extracted, the strict concavity of the function $t\mapsto t^{2/q}$ implies
the convergence.\\
(ii) In the case where $s \in (0,N)$ is an even integer,
one could also use a classical concentration compactness result
of Lions instead of G\'erard's result, see \cite[Corollary 1]{Lions1}.\\     
(iii) For arbitrary $s\in (0,N)$, one could also use the duality
between (\ref{eq:1}) and (\ref{eq:4}) explained in the introduction
and another concentration compactness result of Lions about optimizing
sequences for (\ref{eq:4}), see \cite[Theorem 2.1]{Lions2}. To us it seemed more
natural to use a technique directly applicable to optimizing sequences for (\ref{eq:1}).}
\end{remark}

With the help of Proposition~\ref{sec:proof-remainder-term-1} and Lemma~\ref{sec:proof-remainder-term-2}, we may
now complete the 

\begin{proof}[Proof of Theorem~\ref{maintheorem-reform}] 
Let $u \in
H^{\frac{s}{2}}(\mathbb{S}^{N})$. Since $0 \in \overline {\cM_*}$, we
have $d(u,\cM_*) \le \|u\|_*$. If $d(u,\cM_*) < \|u\|_*$,
then the first inequality in (\ref{eq:17}) follows from
Proposition~\ref{sec:proof-remainder-term-1}, and it is trivially
satisfied if $d(u,\cM_*)= \|u\|_*$. To prove the second inequality in
(\ref{eq:17}) for some $\alpha>0$, we argue by contradiction.  For this we assume that
there exists a sequence $(u_{m})_m$ in
$H^{\frac{s}{2}}(\mathbb{S}^{N})\setminus \overline{\cM_*} $ with
\begin{equation}
  \label{eq:24}
\frac{\|u_m\|_*^2-\cS |u_m|_q^2}{d^{2}(u_{m}, \cM_*)}\rightarrow 0
\quad \text{as $m\rightarrow \infty$.}
\end{equation}
By homogeneity we can assume that $\|u_m\|_*=1$ for all $m \in \N$,
then $d(u_m,\cM_*) \le 1$ for all $m \in \N$ and therefore
(\ref{eq:24}) implies that $\lim \limits_{m \to \infty}|u_m|_q^2
=\frac{1}{\cS}$. Hence Lemma~\ref{sec:proof-remainder-term-2} gives 
$d(u_m,\cM) \to 0$ as $m \to \infty$. But then
Proposition~\ref{sec:proof-remainder-term-1} shows that (\ref{eq:24})
must be false. We conclude that there exists $\alpha>0$ such that 
$$
\|u\|_*^2-\cS |u|_q^2 \ge \alpha\, d^{2}(u,\cM_*) \qquad \text{for all
$u \in H^{\frac{s}{2}}(\mathbb{S}^{N})$,}
$$
as claimed.
\end{proof}

\section{The weak $L^{q/2}$ remainder term inequality for domains of
  finite measure}
\label{sec:weak-lq2-remainder}
In this section we give the proof of Theorem~\ref{maincorollary}. For
this we define 
$$
U_{\lambda,y} \in \H,\qquad U_{\lambda,y}(x):=\lambda U(\lambda^{\frac{2}{N-s}}(x-y))
$$\
for  $c \in \R \setminus \{0\}, \lambda>0$ and $y \in
\R^N$, so that 
$$
\cM= \{c U_{\lambda,y}\::\: c \in \R \setminus \{0\}, \lambda>0, y \in
\R^N \}.
$$
It will be convenient to adjust the notation for the
weak $L^{q/2}$-norm. We fix $q=\frac{2N}{N-s}$ from now on, and we
write 
$$
|u|_{w,\Omega}= \sup_{\stackrel{A \subset \Omega}{|A|>0}}|A|^{-\frac{s}{N}} \int_A |u|\,dx.
$$
for the weak $L^{q/2}$-norm of a measurable function $u$ defined on a
measurable set $\Omega \subset \R^N$. We note the following scaling property, which follows by direct computation: 
\begin{equation}
  \label{eq:33}
|U_{\lambda,y}|_{w,\R^N}=  |U_{\lambda,0}|_{w,\R^N} = \frac{|U|_{w,\R^N}}{\lambda}
 \qquad \text{for $\lambda>0$, $y \in \R^N$.}
\end{equation}
Similarly, for a fixed domain $\Omega \subset \R^N$, $u \in \HO$ and
$\lambda>0$, define   
$$
\Omega_\lambda:= \lambda^{-2/(N-s)}\Omega \,\subset\, \R^N \qquad
\text{and}\qquad u_\lambda \in
\mathring{H}^{\frac{s}{2}}(\Omega_\lambda),\quad u_\lambda(x)= \lambda u(\lambda^{\frac{2}{N-s}}x).
$$ 
Then a direct computation shows 
\begin{equation}
  \label{eq:37}
|\Omega_\lambda|= \lambda^{-q}|\Omega|, \quad
|u_\lambda|_{w,\Omega_\lambda}= \frac{|u|_{w,\Omega}}{\lambda} \quad
\text{and}\quad d(u_\lambda,\cM)=d(u,\cM).
\end{equation}

Theorem~\ref{maincorollary} will follow immediately from the
following Proposition.

\begin{proposition}
\label{sec:weak-lq2-remainder-1}
There exists a constant $C_0$ depending only on $N$ and $s \in (0,N)$ such that
\begin{equation}
  \label{eq:36}
|u|_{w,\Omega}  \le C_0 |\Omega|^{\frac{1}{q}}\, d(u,\cM)   
\end{equation}
for all subdomains $\Omega \subset \R^N$ with $|\Omega|< \infty$
  and all $u \in \HO$.
\end{proposition}

\begin{proof}
By the scaling properties noted in (\ref{eq:37}), it suffices to
consider a subdomain $\Omega \subset \R^N$ with $|\Omega|=1$ in the sequel. In this case we have, by H\"older's inequality and (\ref{eq:1}), 
\begin{equation}
|u|_{w,\Omega} \le \|u\|_{L^q(\Omega)} 
\le \|u\|_{L^q(\R^N)}   
\le \frac{1}{\sqrt{\cS}}  \|u\|_{s/2}\quad \text{for every $u \in \H$.}      \label{eq:34}
\end{equation}
In the following, let $\rho \in (0,1)$ be given by  
\begin{equation}
  \label{eq:39}
\frac{\rho}{\sqrt{\cS}
  (1-\rho)}= \Bigl(|\mathbb{S}^{N-1}| \int_{1}^\infty
\frac{r^{N-1}}{(1+r^2)^{N}}\,dr\Bigr)^{\frac{1}{q}}
\end{equation}
Let $u \in \HO$. If $\rho \|u\|_{s/2} \leq d(u,\cM)$, then 
\begin{equation}
  \label{eq:38}
|u|_{w,\Omega} \le \frac{1}{\rho \sqrt{\cS}} d(u,\cM)  
\end{equation}
as a consequence of (\ref{eq:34}). So in the remainder of this proof we assume that 
\begin{equation}
 \label{eq:case}
\rho \|u\|_{s/2} > d(u,\cM) \,.
\end{equation}
By homogeneity we may assume that $\|u\|_{s/2}=1$. Since $\rho<1$, the infimum in the definition of $d(u,\cM)$ is attained as a consequence of
(\ref{eq:case}), and we have $d(u,\cM)=\| u - c U_{\lambda,y}\|_{s/2}$ for some $c\in\R$, $\lambda>0$ and $y\in\R^n$. Moreover, \eqref{eq:case} implies that
$$
| 1 - c | = \left| \|u\|_{s/2} - \| c U_{\lambda,y} \|_{s/2} \right| \leq d(u,\cM) \leq \rho \,,
$$
that is, $1-\rho\leq c\leq 1+\rho$. We note that 
\begin{align*}
d(u,\cM)^2 & =\|u-c U_{\lambda,y}\|_{s/2}^2 \ge \cS \|u-c U_{\lambda,y}\|_{L^q(\R^N)}^2\\
&\ge \cS |c|^2 \|U_{\lambda,y}\|_{L^q(\R^N\setminus \Omega)}^2 \ge \cS (1-\rho)^2 \|U_{\lambda,y}\|_{L^q(\R^N\setminus \Omega)}^2 \,.
\end{align*}
Now let $B\subset\R^N$ denote the open ball centered at zero with
$|B|=1$, and let $r_0>0$ denote the radius of $B$. Since the function $U$ in (\ref{eq:6}) is radial
and strictly decreasing in the radial variable, the bathtub principle \cite[Theorem 1.14]{LL} implies that
$$
\|U_{\lambda,y}\|_{L^q(\R^N\setminus \Omega)}^2 \ge
\|U_{\lambda,y}\|_{L^q(\R^N \setminus (B+y))}^2=  \|U_{\lambda,0}\|_{L^q(\R^N \setminus B)}^2 \,,
$$
and hence 
\begin{equation}
\|U_{\lambda,0}\|_{L^q(\R^N \setminus B)}^q \le \Bigl(\frac{d(u,\cM)}{\sqrt{\cS}
  (1-\rho)}\Bigr)^q  \le \Bigl(\frac{\rho}{\sqrt{\cS}
  (1-\rho)}\Bigr)^q= |\mathbb{S}^{N-1}| \int_{1}^\infty
\frac{r^{N-1}}{(1+r^2)^{N}}\,dr   \label{eq:22}
\end{equation}
by our choice of $\rho$ in (\ref{eq:39}). On the other hand, we
compute 
$$
\|U_{\lambda,0}\|_{L^q(\R^N \setminus B)}^q = |\mathbb{S}^{N-1}|
\int_{r_0}^\infty
\frac{r^{N-1} \lambda^{q}}{\big[1+(\lambda^{\frac{2}{N-s}}r)^2\big]^N}\,dr = |\mathbb{S}^{N-1}| \int_{\lambda^{\frac{2}{N-s}}r_0}^\infty
\frac{r^{N-1}}{(1+r^2)^{N}}\,dr
$$
This implies that $\lambda^{\frac{2}{N-s}}r_0\geq 1$ and therefore 
\begin{align}
\|U_{\lambda,0}\|_{L^q(\R^N \setminus B)}^q & = |\mathbb{S}^{N-1}| \int_{\lambda^{\frac{2}{N-s}}r_0}^\infty
\frac{r^{N-1}}{(1+r^2)^{N}}\,dr \label{eq:27} \\
& \ge 2^{-N} |\mathbb{S}^{N-1}|
\int_{\lambda^{\frac{2}{N-s}}r_0}^\infty\,
\frac{dr}{r^{N+1}}=\frac{|\mathbb{S}^{N-1}|}{N\, (2r_0)^N}\,\lambda^{-q}.
\nonumber
\end{align}
Combining (\ref{eq:22}) and (\ref{eq:27}), we conclude that 
\begin{equation}
  \label{eq:35}
d(u,\cM) \ge \frac{C_1}{\lambda} \qquad \text{with $C_1:= \sqrt{\cS}(1-\rho)\Bigl(\frac{|\mathbb{S}^{N-1}|}{N\, (2r_0)^N}\Bigl)^{\frac{1}{q}}$.} 
\end{equation}
Using (\ref{eq:33}), (\ref{eq:34}) and (\ref{eq:35}), we find that 
\begin{align*}
|u|_{w,\Omega}& \le |c U_{\lambda,y}|_{w,\Omega} + |u- c U_{\lambda,y}|_{w,\Omega}\le  (1+\rho) |U_{\lambda,y}|_{w,\R^N} + \frac{1}{\sqrt{\cS}}
\|u- c U_{\lambda,y}\|_{s/2}\\
&= \frac{1+\rho}{\lambda}
|U|_{w,\R^N} + \frac{1}{\sqrt{\cS}} d(u,\cM) \le C_2 d(u,\cM)
\end{align*}
with $C_2:= \frac{(1+\rho)}{C_1}|U|_{w,\R^N} +
\frac{1}{\sqrt{\cS}}$. Combining this with (\ref{eq:38}), we thus obtain the claim with $C_0:= \max \{C_2,\frac{1}{\rho \sqrt{\cS}}\}$. 
\end{proof}

Finally, Theorem~\ref{maincorollary} now simply follows by combining
Theorem~\ref{maintheorem} and
Proposition~\ref{sec:weak-lq2-remainder-1} and setting $C:=C_0^{-2}$. 

{\em \bf Acknowledgement.} U.S. National Science Foundation grant
PHY-1068285 (R.F.) and German Science Foundation (DFG) grant WE
2821/4-1 (T.W.) is acknowledged. Shibing Chen wants to thank Robert McCann for
helpful discussions.

\end{document}